\documentclass[11pt,a4paper]{amsart}

\usepackage{epsfig}
\usepackage{graphicx}
\usepackage[usenames,dvipsnames]{xcolor}
\usepackage[colorlinks=true,linkcolor=Blue,citecolor=Green]{hyperref}
\usepackage{amsmath,amsthm,amsfonts,amssymb}
\usepackage{nicefrac}
\usepackage[T1]{fontenc}
\usepackage[utf8]{inputenc}
\usepackage[textwidth=2.5cm,textsize=small]{todonotes}
\usepackage[capitalize]{cleveref}   
\crefname{enumi}{}{}
\usepackage[alwaysadjust]{paralist}

\newtheorem{theorem}{Theorem}[section]

\theoremstyle{definition}

\newtheorem{remark}[theorem]{Remark}

\DeclareMathOperator{\inter}{int}

\DeclareMathOperator{\conv}{conv}

\def\K{\mathcal{K}}

\def\R{\mathbb{R}}

\def\Z{\mathbb{Z}}

\def\N{\mathbb{N}}

\DeclareMathOperator{\vol}{vol}

\DeclareMathOperator{\aff}{aff}
\DeclareMathOperator{\lin}{lin}
\DeclareMathOperator{\LE}{G}

\begin{document}

\title[A general discrete Minkowski theorem]{A generalization of the discrete version of Minkowski's fundamental theorem}
\dedicatory{In memory of Hermann Minkowski on the occasion of his 150th birthday}

\author{Bernardo Gonz\'{a}lez Merino}
\address{Zentrum Mathematik, Technische Universit\"at M\"unchen, Boltzmannstr. 3, 85747 Garching bei M\"unchen, Germany}
\email{bg.merino@tum.de}


\author{Matthias Henze}
\address{Institut f\"ur Informatik, Freie Universit\"at Berlin, Takustra\ss e 9, 14195 Berlin, Germany}
\email{matthias.henze@fu-berlin.de}

\thanks{This research was supported through the program "Research in Pairs" by the Mathematisches Forschungsinstitut Oberwolfach in 2014.
The first author was partially supported by MINECO-FEDER project
reference MTM2012-34037, Spain, and by Consejer\'ia de Industria,
Turismo, Empresa e Innovaci\'on de la CARM through Fundaci\'on
S\'eneca, Agencia de Ciencia y Tecnolog\'ia de la Regi\'on de
Murcia, Programa de Formaci\'on Postdoctoral de Personal
Investigador. The second author was supported by the ESF EUROCORES
programme EuroGIGA-VORONOI, (DFG): RO 2338/5-1.}



\begin{abstract}
One of the most fruitful results from Minkowski's geometric viewpoint on number theory is his so called $1$st Fundamental Theorem.
It provides an optimal upper bound for the volume of a $0$-symmetric convex body whose only interior lattice point is the origin.
Minkowski also obtained a discrete analog by proving optimal upper bounds on the number of lattice points in the boundary of such convex bodies.
Whereas the volume inequality has been generalized to any number of interior lattice points already by van der Corput in the 1930s, a corresponding result for the discrete case remained to be proven.
Our main contribution is a corresponding optimal relation between the number of boundary and interior lattice points of a $0$-symmetric convex body.
The proof relies on a congruence argument and a difference set estimate from additive combinatorics.
\end{abstract}

\maketitle

\section{Introduction}

A \emph{convex body} in the Euclidean vector space $\R^n$, $n\in\N$, is a compact convex set~$K$ whose set of interior points, denoted by $\inter K$, is nonempty.
The \emph{convex hull} of a subset $S\subset\R^n$ is written as $\conv S$.
We say that a convex body is \emph{strictly convex} if its boundary does not contain a proper line segment, and we write $\K^n_o$ for the family of \emph{$0$-symmetric} convex bodies in~$\R^n$, that is, convex bodies~$K$ with $K=-K$, where $tK=\{tx:x\in K\}$, for~$t\in\R$.
Moreover, we use standard terminology from the theory of convex sets, and we refer the reader to~\cite{gruber2007convex} for all the necessary background.

Motivated by the fundamental inequalities of Minkowski and its various generalizations and extensions, we are interested in the relation of a $0$-symmetric convex body to the lattice~$\Z^n$ consisting of all points in $\R^n$ with only integral coordinates.
A point of $\Z^n$ is shortly called \emph{lattice point} in the sequel, and we denote by $\LE(S)=|S\cap\Z^n|$ the number of lattice points contained in a given set $S\subset\R^n$.
Minkowski proved that the cube $C_n=[-1,1]^n$ has maximal \emph{volume} (Lebesgue measure) among all convex bodies in $\K^n_o$ with the property that the origin is their only interior lattice point (see~\cite[\S30]{minkowski1896geometrie} or~\cite[Sect.~22]{gruber2007convex}).
In symbols,
\begin{align}
\vol(K)&\leq\vol(C_n)=2^n,\quad\textrm{for every }K\in\K^n_o\textrm{ with }\Z^n\cap\inter K=\{0\}.\label{eqn_first_minkowski}
\end{align}
This inequality lies at the heart of Minkowski's geometric viewpoint on number theoretical questions.
Its wide applicability, reaching beyond geometry and number theory, inspired the quest for generalizations and analogous relations ever since (see~\cite[Sect.~22]{gruber2007convex} for classic applications and more background information, and see~\cite[Ch.~3]{taovu2006additive} for its connections to additive combinatorics).
Minkowski moreover obtained a discrete version of this fundamental inequality (see~\cite[pp.~79-80]{minkowski1896geometrie} or~\cite[Thm.~30.2]{gruber2007convex}), saying that the cube also maximizes the total number of lattice points in $0$-symmetric convex bodies~$K$ obeying the above condition.
More precisely,
\begin{align}
\LE(K)&\leq\LE(C_n)=3^n,\quad\textrm{for every }K\in\K^n_o\textrm{ with }\Z^n\cap\inter K=\{0\},\label{eqn_first_mink_disc}
\end{align}
and
\begin{align}
\LE(K)&\leq2^{n+1}-1,\quad\textrm{if }K\textrm{ is moreover strictly convex}.\label{eqn_first_mink_disc_sc}
\end{align}
It has been shown in~\cite{draismamcallisternill2012lattice} that equality holds in~\eqref{eqn_first_mink_disc} if and only if~$K$ is \emph{unimodularly equivalent} to the cube~$C_n$.
That is, there exists an invertible matrix $A\in\Z^{n\times n}$ whose inverse is also an integral matrix and such that $K=AC_n$, where for a matrix $A\in\R^{n\times n}$ and a subset $S\subseteq\R^n$ we write $AS=\{As:s\in S\}$.
A suitable smoothing of the convex hull of $[0,1]^n\cup[-1,0]^n$ shows that the inequality~\eqref{eqn_first_mink_disc_sc} is also best possible.
Besides Minkowski's original monograph~\cite{minkowski1896geometrie}, the book by Gruber \& Lekkerkerker~\cite{gruberlekker1987geometry}, in particular Sections 9.4, 26.2 and the Supplements to Chapter 4, is an excellent reference for the theory that developed out of these results.
More recent developments are covered in~\cite{gruber2007convex}.

Another way of saying that $K\in\K^n_o$ contains only the origin as an interior lattice point is that its \emph{first successive minimum}
\[\lambda_1(K)=\min\left\{\lambda>0:\lambda K\cap\Z^n\neq\{0\}\right\}\]
is at least one.
Based on this concept, Betke, Henk \& Wills~\cite{betkehenkwills1993successive} proved an extension of Minkowski's inequalities:
\begin{align}
\LE(K)&\leq\left\lfloor\frac2{\lambda_1(K)}+1\right\rfloor^n,\quad\textrm{for every }K\in\K^n_o,\label{eqn_BHW}
\end{align}
and
\begin{align}
\LE(K)&\leq2\left\lceil\frac2{\lambda_1(K)}\right\rceil^n-1,\quad\textrm{if }K\textrm{ is moreover strictly convex}.\label{eqn_BHW_sc}
\end{align}
Here, the floor function $\lfloor x\rfloor$ and the ceiling function $\lceil x\rceil$ of a real number~$x$ denote, as usual, the largest integer smaller than or equal to $x$, and the smallest integer bigger than or equal to $x$, respectively.
Clearly,~\eqref{eqn_first_mink_disc} and~\eqref{eqn_first_mink_disc_sc} are special cases of~\eqref{eqn_BHW} and~\eqref{eqn_BHW_sc}, respectively.
Moreover, using $\lambda_1(tK)=\frac1{t}\lambda_1(K)$, for any $t>0$, and $\vol(K)=\lim_{t\to\infty}\LE(tK)/t^n$, a limit argument shows that~\eqref{eqn_first_minkowski} is a consequence of~\eqref{eqn_BHW}.

Another perspective on extending Minkowski's volume inequality has already been taken by van der Corput~\cite{vdcorput1936verallg}, who showed that
\begin{align}
\vol(K)&\leq 2^{n-1}\left(\LE(\inter K)+1\right),\quad\textrm{for every }K\in\K^n_o.\label{eqn_vandercorput}
\end{align}
Equality is attained for the stretched cube $C_{n-1}\times[-\ell,\ell]$, where $\ell\in\N$.
In this work, we are interested in a discrete version of van der Corput's result, that is, an upper bound on~$\LE(K)$ in terms of the number of interior lattice points in~$K\in\K^n_o$.
In order to see that such a bound exists, we observe that for any $0$-symmetric convex body~$K$, and any of its lattice points~$z\in K\cap\Z^n$, the open line segment $(-z,z)$ contains at most $\LE(\inter K)$ lattice points.
Therefore, $|[0,z)\cap\Z^n|\leq(\LE(\inter K)+1)/2$ and thus
\begin{align}
\frac2{\lambda_1(K)}\leq\LE(\inter K)+1,\quad\textrm{for every }K\in\K^n_o.\label{eqn_lambda1_Gint}
\end{align}
Combining this inequality with~\eqref{eqn_BHW} gives $\LE(K)\leq\left(\LE(\inter K)+2\right)^n$.
A connection to a conjecture of Betke, Henk \& Wills~\cite[Conj.~2.1]{betkehenkwills1993successive} that generalizes~\eqref{eqn_BHW} provides a hint on how a best possible such inequality may look like.
These authors claim that
\[\LE(K)\leq\prod_{i=1}^n\left\lfloor\frac2{\lambda_i(K)}+1\right\rfloor,\quad\textrm{for every }K\in\K^n_o,\]
where $\lambda_i(K)=\min\{\lambda>0:\lambda K\cap\Z^n\textrm{ contains }i\textrm{ linearly independent points}\}$ is the $i$th successive minimum of~$K$ (see~\cite{malikiosis2012adiscrete} for the state of the art regarding this conjecture).
If this were true and we moreover assume that $\lambda_2(K)\geq1$, that is, the interior lattice points of $K$ are collinear, then by~\eqref{eqn_lambda1_Gint}, we have \[\LE(K)\leq3^{n-1}\left\lfloor\frac2{\lambda_1(K)}+1\right\rfloor\leq3^{n-1}\left(\LE(\inter K)+2\right).\]
As our main result, we prove that this bound holds unconditionally, and thus obtain an exact discrete version of van der Corput's inequality which includes~\eqref{eqn_first_mink_disc} as special case.

\begin{theorem}\label{thm_main_discrete_mink}
Let $K\in\K^n_o$. We have
\begin{align}
 \LE(K)&\leq3^{n-1}\left(\LE(\inter K)+2\right),\label{eqn_main_theorem}
\end{align}
and equality is attained if and only if $K$ is unimodularly equivalent to the parallelepiped $C_{n-1}\times[-\ell,\ell]$, for some $\ell\in\N$.
\end{theorem}
Note that an inequality of Scott~\cite{scott1976on} implies the bound~\eqref{eqn_main_theorem} in the case $n=2$.
Our proof of \cref{thm_main_discrete_mink} is based on two main ingredients.
The first is an application of an elegant congruence argument that lies behind many pertinent results in the geometry of numbers (cf.~\cite{betkehenkwills1993successive,rabinowitz1989a,vdcorput1936verallg} or~\cite[Sect.~30/31]{gruber2007convex}).
We say that two lattice points $x,y\in\Z^n$ are \emph{congruent modulo $m\in\Z$}, if $x-y\in m\Z^n$.
Observe that the points of $\Z^n$ are partitioned into precisely~$m^n$ congruence classes, also often called \emph{residue classes}.
In order to illustrate the method, we recall Minkowski's proof of~\eqref{eqn_first_mink_disc}:
Assume that for some $K\in\K^n_o$, we have $\LE(K)>3^n$.
Then there are $x,y\in K\cap\Z^n$, $x\neq y$, that are congruent modulo~$3$.
By symmetry and convexity of $K$, this shows that $(x-y)/3$ is a non-zero interior lattice point of $K$, contradicting the assumptions on the body.
The second ingredient is an estimate on the size of difference sets of non-collinear finite point sets to which the next section is devoted.
The details for \cref{thm_main_discrete_mink} are then carried out in~\cref{sect_proof_disc_Mink}.

In the case of strictly convex bodies, the congruence argument alone leads to a bound that generalizes~\eqref{eqn_first_mink_disc_sc} as follows:
\begin{align}
\LE(K)\leq 2^n(\LE(\inter K)+1)-1,\quad\textrm{for every strictly convex }K\in\K^n_o.\label{eqn_osymm_theorem_sc}
\end{align}
However, this inequality can be improved by roughly a factor of $2/3$.
Moreover, the condition of $0$-symmetry on the involved convex body can be removed.
This is a curious phenomenon that deviates from the general case as a comparison of~\eqref{eqn_main_theorem} and the inequality~\eqref{eqn_pikhurko_lp} below shows.

\begin{theorem}\label{thm_main_theorem_sc}
Let $K\in\K^n$ be strictly convex. We have
\begin{align}
\LE(K)\leq 2(2^{n-1}-1)\left\lceil\frac23(\LE(\inter K)+1)\right\rceil+\LE(\inter K)+2.\label{eqn_main_theorem_sc}
\end{align}
\end{theorem}

For $K\in\K^n$ with $\LE(\inter K)=1$, also this bound reduces to Minkowski's inequality~\eqref{eqn_first_mink_disc_sc}.
The proof of \cref{thm_main_theorem_sc} is based on a connection to Helly numbers of families of $S$-convex sets.
We elaborate on this in~\cref{sect_helly_number_connection}, where we also give the details for the bound~\eqref{eqn_osymm_theorem_sc}.

We end this introduction by shortly discussing the state of the problem of obtaining sharp estimates analogous to~\eqref{eqn_vandercorput} and our~\cref{thm_main_discrete_mink} for the case of not necessarily $0$-symmetric convex bodies.
Such investigations already started in the 1980s and culminated in the work of Pikhurko~\cite{pikhurko2001lattice}, who proved
\begin{align}
\vol(P)&\leq(8n)^n 15^{n2^{2n+1}}\LE(\inter P),\label{eqn_pikhurko}
\end{align}
and
\begin{align}
\LE(P)&\leq n!(8n)^n 15^{n2^{2n+1}}\LE(\inter P)+n,\label{eqn_pikhurko_lp}
\end{align}
whenever $\Z^n\cap\inter P\neq\emptyset$ and $P$ is a \emph{lattice polytope} in $\R^n$, that is, the convex hull of finitely many lattice points.
Although the minimal factor in front of $\LE(\inter P)$ admitting inequalities of this type is known to be doubly exponential in~$n$, the above bounds are assumed to be far from tight.
The determination of the exact bound is only solved for~\eqref{eqn_pikhurko} in the case of lattice simplices with exactly one interior lattice point~\cite{averkovkruempelnill2014largest} (cf.~\cite{pikhurko2001lattice} for further information on the history of the problem).

\section{The equality case in a planar difference set estimate}\label{sect_difference_set_estimate}

In this section, we discuss a combinatorial result on the minimal number of difference vectors generated by a \emph{non-collinear} point set, by which we mean a point set that is not contained in a line.
To this end, for $U,V\subset\R^n$, we write $U+V=\{u+v:u\in U,v\in V\}$.
A set of the form $\{u,u+s,\ldots,u+(k-1)s\}$, for some $u,s\in\R^n$ and $k\in\N$, is called an \emph{arithmetic progression}, and the length of the vector~$s$ is its \emph{spacing}.
It is easy to see that for any nonempty finite sets $U,V\subset\R^n$, we have
\begin{align}
|U+V|\geq|U|+|V|-1.\label{eqn_sumset_estimate}
\end{align}
Equality is attained if and only if $U$ and $V$ are arithmetic progressions with the same spacing.
The case $V=-U$ of this inequality is of particular interest.
Freiman, Heppes~\& Uhrin~\cite{freimanheppesuhrin1990alower} showed that if we assume that $U$ has affine dimension $d$, then one can improve the above bound to
\begin{align}
|U-U|&\geq(d+1)|U|-\binom{d+1}2.\label{eqn_FHU}
\end{align}
Note that the authors of~\cite{freimanheppesuhrin1990alower} apply this inequality to sharpen a classic result of Blichfeldt on the number of lattice points in the difference set of an arbitrary Lebesgue-measurable set.
For $d=1,2$, the estimate above cannot be further improved, but it is conjectured that for any $d\geq3$ there is a better bound.
In fact, revising a conjecture of Freiman, Stanchescu~\cite{stanchescu2001an} claims that, for every $d\geq2$, the maximal factor in front of~$|U|$ in an inequality of the type~\eqref{eqn_FHU} is given by $2(d-1)+1/(d-1)$ and proves this for the case $d=3$.
Such difference set estimates embed in the currently very active field of additive combinatorics, where people study more generally the structure of subsets~$U$ of some abelian group whose sum-sets or difference-sets $U\pm U$ have either very small or very large cardinality.
For instance, generalizing an earlier result by Freiman, Ruzsa found an optimal lower bound on $|U+V|$, for two given subsets $U$ and $V$, which includes~\eqref{eqn_FHU} as a special case (see~\cite{ruzsa2006additive} for a survey on this and related problems).
The interested reader may also consult the book of Tao \& Vu~\cite{taovu2006additive} that covers the recent developments and their various applications in many branches of mathematics.

For our purposes, we need to investigate the case $d=2$ of the inequality~\eqref{eqn_FHU} more closely.
We have seen that Freiman, Heppes \& Uhrin obtained the optimal lower bound on the size of the difference set in this case.
Moreover, for the case that $|U|$ is even, Stanchescu~\cite{stanchescu1998on} characterized the point sets~$U$ attaining equality.
However, in order to be able to prove \cref{thm_main_discrete_mink}, we also need to characterize the point sets of odd size with minimal value of $|U-U|$.
To the best of our knowledge this has not been worked out before, and thus we give the complete proof of all three statements for the readers convenience.
Before we can state the result, we need to introduce a notion of a generalized arithmetic progression.
We say that a point set $U\subset\R^n$ is an \emph{arithmetic progression of type $(k,l)$} if there exists an anchor point $u\in\R^n$ and two linearly independent vectors $s,t\in\R^n$ such that $U=U'\cup(U'+t)\cup\ldots\cup(U'+(l-1)t)$, where $U'=\{u,u+s,\ldots,u+(k-1)s\}$.
Moreover, we say that a point set $U\subset\R^n$ is an \emph{incomplete arithmetic progression of type~$(k,l)$}, if there exists an $x\in\R^n$ such that $U\cup\{x\}$ is an arithmetic progression of type~$(k,l)$ and $x$ is a vertex of $\conv\{U\cup\{x\}\}$.

\begin{theorem}\label{thm_number_vectors}
For any set $U\subset\R^n$, $n\geq2$, of $k\in\N$ non-collinear points, we have
\[|U-U|\geq\begin{cases}3k-3&\textrm{if }k\textrm{ is even,}\\ 3k-2&\textrm{if }k\textrm{ is odd.}\end{cases}\]
Equality holds for even $k$ if and only if $U$ is an arithmetic progression of type $(k/2,2)$, and for odd $k$ if and only if $U$ is either an incomplete arithmetic progression of type $(\lceil k/2 \rceil,2)$ or an arithmetic progression of type $(3,3)$.
\end{theorem}
\begin{proof}
First of all, we show that it suffices to prove the theorem for subsets~$U$ of $\R^2$.
Since $U$ is finite, we can always find a two-dimensional linear subspace~$S$ of $\R^n$ such that the orthogonal projection $U_S$ of $U$ onto $S$ is a non-collinear point set with the same cardinality as~$U$.
Indeed, any two-dimensional linear subspace such that no direction vector $v-w$, for $v,w\in U$, is orthogonal to it, will do.
Clearly, $|U-U|\geq|U_S-U_S|$ and hence it suffices to prove the bound for~$U_S$.
If~$U$ has affine dimension at least three, taking a little more care we can find such a subspace~$S$ with the additional property that two difference vectors of $U-U$ are projected onto the same difference vector in~$U_S-U_S$.
We leave the details of this construction to the reader.\footnote{An alternative way to reduce the proof to two-dimensional point sets~$U$ is via the inequality~\eqref{eqn_FHU}. For any $d\geq3$ the implied bound is strictly greater than what is claimed in the theorem.}
Therefore, any set~$U$ attaining equality in any of the claimed bounds has affine dimension at most two.

For the remainder of the proof we assume that $U\subset\R^2$, and we closely follow the arguments in~\cite[Sect.~3]{stanchescu1998on} with adjustments and extensions wherever necessary.
Let $L$ be a supporting line of an edge of $\conv U$ and let $L'$ be the supporting line parallel to~$L$ and on the other side of $\conv U$.
Since~$U$ is non-collinear, we have $L\neq L'$.
Without loss of generality, we let $\ell=|U\cap L|\geq|U\cap L'|=m\geq1$, and by construction, we have~$\ell\geq2$.

\begin{figure}[t]
\includegraphics[scale=.85]{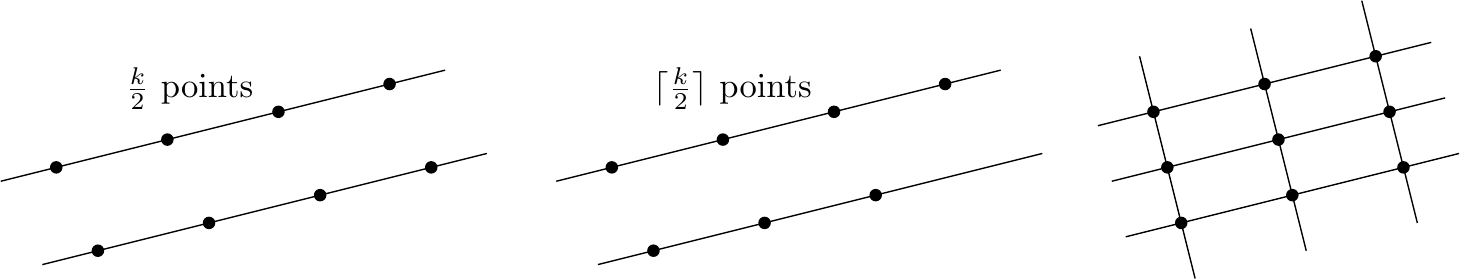}
\caption{The equality cases in \cref{thm_number_vectors}, for $k$ even (left) and $k$ odd (middle and right).}
\label{fig_equality_cases}
\end{figure}

We now prove the claimed inequalities together with their characterization of equality by induction on $k-\ell-m$.
The base case is $k-\ell-m=0$ which means that the point set~$U$ is contained in the two parallel lines~$L$ and~$L'$.
Using twice the inequality~\eqref{eqn_sumset_estimate} yields
\begin{align*}
|U-U|&\geq|(U\cap L)-(U\cap L)|+\big|\pm\big((U\cap L)-(U\cap L')\big)\big|\\
&\geq2\ell-1+2(\ell+m-1)=3k+(\ell-m)-3.
\end{align*}
It follows that $|U-U|\geq3k-3$, for arbitrary~$k$, and $|U-U|\geq3k-2$, if $k=\ell+m$ is odd.
Equality is attained for even~$k=|U|$, if and only if $\ell=m$, the points $U\cap L$ form an arithmetic progression and the point set $U\cap L'$ is a translate of $U\cap L$.
That is, $U$ is an arithmetic progression of type $(k/2,2)$.
If~$k$ is odd, equality holds if and only if $\ell=m+1$, the points $U\cap L$ form an arithmetic progression and every difference vector generated by $U\cap L'$ can be generated by $U\cap L$.
It is not very difficult to see that the latter condition holds if and only if $U\cap L'$ is a translate of $(U\cap L)\setminus\{p\}$, for some endpoint $p\in U\cap L$.
Hence,~$U$ is an incomplete arithmetic progression of type $(\lceil k/2 \rceil,2)$.

To carry out the induction step, we now assume $k-\ell-m>0$, that is, $U\neq(U\cap L)\cup(U\cap L')$, and we write $U'=U\setminus(U\cap L')$.
As before, there are at least $2(\ell+m-1)$ vectors of the form $\pm(v-w)$, where $v\in U\cap L$ and $w\in U\cap L'$.
This number is at least $4m-2\geq3m$, for $m\geq2$, and at least $2\ell\geq4>3m$, for $m=1$.
Moreover, we find exactly $3m$ such difference vectors if and only if $\ell=m=2$, and the two points in $U\cap L$ have the same spacing as the two points in $U\cap L'$.
By construction, the orthogonal projection of these difference vectors $\pm(v-w)$ onto the line orthogonal to~$L$ is different from the projection of a difference vector generated by any two points of~$U'$.
Thus, apart from vectors of $U'-U'$, the set $U-U$ contains at least~$3m$ further vectors.
Since the points in $U'$ are non-collinear, we inductively get that
\begin{align}
|U-U|&\geq|U'-U'|+3m\geq3(k-m)-3+3m=3k-3.\label{eqn_comb_bound}
\end{align}
If $k=|U|$ is odd, then $3k-3$ is even and so the bound can be improved to $|U-U|\geq3k-2$, as $|U-U|$ is an odd number.

In order to finish the characterization of equality, we need to distinguish some cases.

\medskip
\noindent{\it Case 1: $k$ is even and $|U-U|=3k-3$.}

This holds precisely, if we have $|U'-U'|=3(k-m)-3$ and there are exactly $3m$ difference vectors of the form $\pm(v-w)$, where $v\in U\cap L$ and $w\in U\cap L'$.
By induction hypothesis, the first condition implies that $k-m$ is even and that~$U'$ is an arithmetic progression of type $((k-m)/2,2)$.
Moreover, from the second condition we get $\ell=m=2$ and the spacing of the two points $U\cap L$ is the same as that of the points $U\cap L'$.
Let~$S$ and $S'$ be two parallel lines that contain the set~$U'$, and such that, say, one of the points of~$U\cap L$ is in~$S$ and the other point in~$S'$.
Without loss of generality we assume that these lines are horizontal, with~$S$ being the lower one.
Since~$L$ supports an edge of $\conv U$ and $|U\cap L|=2$, it supports one of the edges of $\conv U'$ that contain exactly two points of~$U$.

We claim that one of the two points of~$U\cap L'$ is in~$S$ and the other one in~$S'$.
In order to see this, let us assume that $w\in U\cap L'$ does not lie on any of the lines~$S$ and~$S'$.
Since~$U$ is not completely contained in the lines~$L$ and~$L'$, there is another line $L''$ parallel to $L$ that is different from these two and contains two points from~$U'$.
If~$w$ lies below~$S$, we let $z$ be the point in $U'\cap L''\cap S'$, and otherwise we let $z$ be the point in $U'\cap L''\cap S$.
But then the difference vector $w-z$ can neither be generated by $U'$ nor by points from $U\cap L$ and $U\cap L'$, contradicting the equality assumption.
Therefore, $w\in S\cup S'$, say $w\in S$, and moreover, it is easy to see that the distance from the closest point in $U'\cap S$ to $w$ must be the same as the (equal) distance of any neighboring points in $U'\cap S$.
Hence,~$U$ is an arithmetic progression of type~$(k/2,2)$.

\medskip
\noindent{\it Case 2: $k$ is odd and $|U-U|=3k-2$.}

In view of~\eqref{eqn_comb_bound}, there are two options.

\medskip
\noindent{\it Case 2.1: $|U'-U'|=3(k-m)-3$ and there are exactly $3m+1$ difference vectors of the form $\pm(v-w)$, where $v\in U\cap L$ and $w\in U\cap L'$.}

The first condition implies by induction that $k-m$ is even, thus $m$ is odd, and that $U'$ is an arithmetic progression of type $((k-m)/2,2)$.
We define the horizontal lines $S$ and $S'$ as in Case 1.
In general, there are at least $2(\ell+m-1)$ difference vectors $\pm(v-w)$, which is exactly $3m+1$ if either $\ell=m=3$, or $\ell=2$ and $m=1$.
Since $L$ supports an edge of $\conv U'$, we see that in the first case either $L=S$ or $L=S'$, and thus $k-m=6$.
Therefore,~$U$ is a set of nine points equally distributed on three parallel lines and a similar argument as in Case 1 above shows that $U$ must be an arithmetic progression of type $(3,3)$ in order to avoid additional difference vectors besides those generated by two points in $U'$ or one point each from $U\cap L$ and $U\cap L'$.
In the second case, that is, $\ell=2$ and $m=1$, the line~$L$ intersects $U'$ in exactly two points and thus supports a short edge of $\conv U'$.
Again, we argue similarly as in Case 1 and we get that the point $U\cap L'$ must lie on one of the horizontal lines $S$ and $S'$, say it lies on~$S$, having the same distance from the closest point in $U'\cap S$ as the (equal) distance of any neighboring points in $U'\cap S$.
Hence, $U$ is an incomplete arithmetic progression of type $(\lceil k/2 \rceil,2)$.

\medskip
\noindent{\it Case 2.2: $|U'-U'|=3(k-m)-2$ and there are exactly $3m$ difference vectors of the form $\pm(v-w)$, where $v\in U\cap L$ and $w\in U\cap L'$.}

The first condition implies by induction that $k-m$ is odd and that $U'$ is either an incomplete arithmetic progression of type $(\lceil(k-m)/2\rceil,2)$ or an arithmetic progression of type $(3,3)$.
From the second condition we infer that $\ell=m=2$, and since $L$ supports an edge of $\conv U'$ and contains exactly two points of $U$, the set $U'$ cannot be an arithmetic progression of type $(3,3)$.
Hence, writing $k'=\lceil(k-m)/2\rceil$, we have that $U'=V\cup(V'+t)$, where $V=\{u,u+s,\ldots,u+(k'-1)s\}$ and $V'=V\setminus\{u+(k'-1)s\}$, for suitable $u,s,t\in\R^2$.
Moreover, it is no restriction to assume that $L$ intersects~$U'$ in the points $u,u+t$, and we let $U\cap L'=\{v,w\}$.
Again by the same argumentation as in Case 1, we see that $v$ and $w$ must be contained in the lines spanned by $V$ and $V'$, respectively, and moreover $v=u+k's$ and $w=u+t+k's$, or vice versa.
But now we find that the pair of difference vectors $\pm(w-(u+s))=\pm(t+(k'-1)s)$ can neither be generated by~$U'$ nor by points from $U\cap L$ and $U\cap L'$, contradicting that $U$ is a point set attaining equality.
Eventually, this shows that the current case cannot occur, and thus finishes our proof.
\end{proof}

\section{Proof of the general discrete Minkowski theorem}\label{sect_proof_disc_Mink}

For the following proof we need the notion of an \emph{$i$-dimensional lattice plane} by which we mean any affine subspace of~$\R^n$ of the form $z+\lin\{z_1,\ldots,z_i\}$, for some linearly independent $z_1,\ldots,z_i\in\Z^n$ and some~$z\in\Z^n$.

\begin{proof}[Proof of~\cref{thm_main_discrete_mink}]
The proof splits up into two main cases.

\medskip
\noindent{\it Case 1: The interior lattice points of~$K$ are contained in a line $L$.}

By the $0$-symmetry of~$K$, the number of interior lattice points in
$K$ is odd, hence $\LE(\inter K)=2t-1$, for some $t\in\N$. Any two
points $v,w\in K\cap\Z^n$ from the same residue class modulo $3$
give rise to a lattice point $(v-w)/3$. Since~$K$ is convex and
$0$-symmetric, this lattice point belongs to the interior of~$K$ and
hence lies on~$L$. This implies that all lattice points of $K$
belonging to the same residue class modulo $3$ are contained in a
line parallel to $L$. Each lattice line contains lattice points from
exactly three different residue classes modulo~$3$. Now, let $L'$ be
a parallel lattice line to~$L$ and let $R_1$, $R_2$, and $R_3$, be
the set of lattice points of $K$ that belong, respectively, to the
three residue classes in $L'$. If one of these sets, say $R_1$,
contains two different lattice points $v,w\in K\cap L'$, then the
segment $[v,w]$ contains at least two lattice points in its relative interior
and in fact one from each set $R_2$ and $R_3$. Therefore, $R_1\cup
R_2\cup R_3\subset L'$. Now, $L'\cap K$ can contain no more than
$2t+1$ lattice points, as we otherwise would get more than $2t-1$
interior lattice points of $K$ by its $0$-symmetry. Hence $|R_1\cup
R_2\cup R_3|\leq 2t+1$. If, on the contrary, $|R_i|\leq 1$, for all
$i=1,2,3$, then clearly $|R_1\cup R_2\cup R_3|\leq 3\leq 2t+1$.

There are $3^{n-1}$ triples of residue classes with points in the
same parallel line to~$L$. By counting the lattice points in $K$ by
containment in these triples, we get
$\LE(K)\leq3^{n-1}\left(2t+1\right)=3^{n-1}\left(\LE(\inter
K)+2\right)$ as claimed.

\medskip
\noindent{\it Case 2: The interior lattice points of~$K$ are non-collinear.}

Let $k\in\N$ be such that $3k-2\leq \LE(\inter K)\leq 3k$. Since
$3k-2\leq\LE(\inter K)$, assuming that $\LE(K)\leq 3^nk$
implies~\eqref{eqn_main_theorem}. Therefore, we suppose that $\LE(K)>3^nk$. By this
assumption, there exists a residue class modulo $3$ that contains at
least $k+1$ different elements $u_0,u_1,\ldots,u_k\in K\cap\Z^n$.
The idea of the proof consists of distinguishing two
subcases, depending on the distribution of these lattice points.
In each subcase, we prove that either \eqref{eqn_main_theorem} is true
or that $K$ contains at least $3k+1$ interior lattice points.
The latter contradicts the choice of $k$ and thus proves that the
respective subcase cannot occur.

\medskip
\noindent{\it Case 2.1: The points $u_0,\ldots,u_k$ are collinear.}

We assume without loss of generality that the points
$u_0,\ldots,u_k$ lie in this sequence on the line~$L$. Since they
belong to the same residue class modulo~$3$, there are at least two
lattice points between any pair $u_i$ and $u_{i+1}$ on~$L$, and
hence the line segment $[u_0,u_k]$ contains at least $3k+1$ lattice
points. Let~$L_0$ be the line parallel to~$L$ passing through the
origin. By $0$-symmetry of~$K$, we see that the central slice $K\cap
L_0$ of~$K$ contains at least $3k-1$ interior lattice points.
Additionally, we find at least one interior lattice point of $K$
outside $L_0$, since the interior lattice points of $K$ are assumed
to be non-collinear. The $0$-symmetry of $K$ implies that the
opposite of this point is an interior lattice point as well, and
thus $\LE(\inter K)\geq3k+1$.

\medskip
\noindent{\it Case 2.2: The points $u_0,\ldots,u_k$ are non-collinear.}

By \cref{thm_number_vectors} there are, depending on the parity of $k$, at least $3k$ or $3k+1$ difference vectors of the form $u_i-u_j$, $i,j\in\{0,1,\ldots,k\}$.
Since the~$u_i$ belong to the same residue class modulo~$3$, the points $(u_i-u_j)/3$ are interior lattice points of~$K$.
Hence, if $k$ is even, we obtain $\LE(\inter K)\geq 3k+1$.
If, on the contrary, $k$ is odd and \cref{thm_number_vectors} gives us exactly $3k$ difference vectors, then its equality characterization shows that $U=\{u_0,\ldots,u_k\}$ is an arithmetic progression of type $((k+1)/2,2)$.

In view of $\LE(\inter K)\geq3k$ (thus having indeed $\LE(\inter K)=3k$),
the desired estimate~\eqref{eqn_main_theorem} holds if $K$ contains not too many more lattice points than $3^nk$.
In fact, if $\LE(K)\leq3^nk+2\cdot3^{n-1}-1$, then $\LE(K)\leq3^{n-1}\LE(\inter K)+2\cdot3^{n-1}-1<3^{n-1}\left(\LE(\inter K)+2\right)$.
Hence, we may assume that $\LE(K)>3^nk+2\cdot3^{n-1}-1$.
Moreover, no residue class contains $k+2$ lattice points of~$K$, since otherwise we get $\LE(\inter K)\geq3k+3$ by the argument in the previous paragraph.
If $c$ denotes the number of residue classes with precisely $k+1$ elements in $K$, then we have $\LE(K)\leq c(k+1)+(3^n-c)k=3^nk+c$ and hence $c\geq2\cdot3^{n-1}$.
Let $U_i$, $i=1,\ldots,c$, be the set of lattice points of~$K$ lying in the $i$th such residue class.

By the above considerations this implies $|U_i-U_i|=3k$, for all $i=1,\ldots,c$, which means that all the $U_i$, $i=1,\ldots,c$, are arithmetic progressions of type $((k+1)/2,2)$.
We claim that under this condition the sets~$U_1,\ldots,U_c$ need to be translates of each other.

To this end, we write $U_i=\{u_0^i,\ldots,u_k^i\}$, for
$i=1,\ldots,c$. More precisely, there are anchor points
$a^i \in \R^n$ and pairs of linearly independent vectors
$s^i, t^i \in \R^n$ such that $u_l^i=a^i+lt^i$ and
$u_{(k+1)/2+l}^i=a^i+s^i+lt^i$, for $l=0,\ldots,(k-1)/2$.
From this explicit description one derives $U_i-U_i=\left\{0,\pm
t^i,\ldots,\pm\frac{k-1}2t^i\right\}\cup\left(\pm s^i+\left\{0,\pm
t^i,\ldots,\pm\frac{k-1}2t^i\right\}\right)$, that is, the
difference vectors in $U_i-U_i$ are equally distributed on three
parallel lines. In order for all of them to generate the same set of
interior lattice points of~$K$, we need to have $U_i-U_i=U_j-U_j$,
for $i,j=1,\ldots,c$, which readily implies $s^i=s^j$ and $t^i=t^j$,
for $i,j=1,\ldots,c$, and hence our claim.

The assumptions that $k$ is odd and $u_0,\ldots,u_k$ are
non-collinear imply that $k\geq 3$. Thus, for every $U_i$, its
affine hull $S_i=\aff U_i$ is a lattice plane. Let~$S$ be the linear
$2$-dimensional subspace parallel to the planes $S_i$,
$i=1,\ldots,c$. Let $w_1,\ldots,w_9\in\Z^n$ be distinct
representatives of the nine residue classes present in $S$.
Moreover, let $r_1,\ldots,r_{3^{n-2}}\in\Z^n$ be such that
$\Z^n=\bigcup_{k,\ell}(r_k+w_\ell+3\Z^n)$ is the partition of $\Z^n$
by the residue classes modulo $3$. Let $r_{k_i}+w_{\ell_i}$ be the
representative of the residue class corresponding to $U_i$. Observe
that if $r_{k_i}=r_{k_j}$, for some $i,j\in\{1,\ldots,c\}$, then
$S_i=S_j$, because $S_i$ actually contains at least one lattice
point of~$K$ from each of the nine residue classes that are present
in~$S_i$. Since $c\geq2\cdot3^{n-1}$, the pigeonhole principle
implies that without loss of generality $r_{k_1}=\ldots=r_{k_6}$,
for $k_1,\ldots,k_6\in\{1,\ldots,3^{n-2}\}$. Therefore,
$S_1=\ldots=S_6$.

Remember that any segment whose endpoints are
different points of $U_i$ contains another two points from
different residue classes modulo 3.
Remember as well that $U_1,\ldots,U_6$ are translates of each other.
The points $a^1+jt^1+ls^1$, $j,l=0,1/3,2/3$, are lattice points in $K$,
each in a different residue class. Since six of these points are
contained in $U_1 \cup \ldots \cup U_6$, either there exists
$j\in\{0,1/3,2/3\}$, such that $a^1+jt^1+ls^1$, $l=1/3,2/3$, belong to $U_2 \cup\ldots \cup U_6$,
or $a^1+jt^1$, $j=0,1/3,2/3$, belong to $U_1 \cup\ldots \cup U_6$.
In the first case, assuming without loss of generality
that $u^2_r=a^1+jt^1+(1/3)s^1$ and $u^3_s=a^1+jt^1+(2/3)s^1$, we find that
$u^2_{r\pm(k+1)/2}=a^1+u^2_0\pm s^1$ and $u^3_{s\pm(k+1)/2}=a^1+u^3_0\pm s^1$
(where each $\pm$ is either $+$ or $-$ depending on $r,s$
taking values in $\{0,1,\ldots,(k-1)/2\}$ or $\{(k+1)/2,(k+3)/2,\ldots,k\}$, respectively), together with
$a^1+jt^1+s^1=(1-j)(a^1+s^1)+j(a^1+t^1+s^1)$ and $a^1+jt^1=(1-j)(a^1+0)+j(a^1+t^1)$, belong to $K$
and they all lie in the same line.
In the second case, again without loss of generality,
the points $u^j_i$, $j=1,2,3$, $i=0,\dots,(k-1)/2$, are
lattice points of $K$ lying in a line.
Therefore, we have shown that there is a lattice line either parallel to
$L_1=\lin\{s^1\}$ containing at least $6$ points of
$U_1\cup\ldots\cup U_6$, or parallel to
$L_2=\lin\{t^1\}$ containing at least
$3(k-1)/2+3$ lattice points of $U_1\cup\ldots\cup U_6$.
By $0$-symmetry of~$K$, either the central slice $C_1$ of~$K$ parallel
to $L_1$ contains at least~$5$ interior lattice points, or the
central slice $C_2$ of~$K$ parallel to $L_2$ contains at least
$3(k-1)/2+1$ interior lattice points. On the other hand, only $3$
lattice points in $C_1$ and only $k$ lattice points in $C_2$ are
derived from $U_1-U_1$. As $3(k-1)/2+1>k$ for every $k\geq3$, this
contradicts the fact that $\LE(\inter K)=3k$ under
the condition $\LE(K)>3^nk+2\cdot3^{n-1}-1$.

Summarizing our investigations, we have seen
that either~\eqref{eqn_main_theorem}
holds true or $\LE(\inter K)\geq3k+1$. As
explained in the beginning of Case 2, this finishes the proof of the
desired inequality~\eqref{eqn_main_theorem}.

\medskip
\noindent{\it Characterization of the equality case.}

For the characterization of the equality case, we consider
$K\in\K^n_o$ such that $\LE(K)=3^{n-1}\left(\LE(\inter K)+2\right)$.

\medskip
\noindent{\it Case 1: The interior lattice points of $K$ are non-collinear.}

The previous paragraph shows that in order for equality to hold there needs to be some odd~$k$ such that $\LE(\inter K)=3k-2$, and in particular $\LE(K)=3^nk$.
Moreover, $k\geq3$ due to the assumed non-collinearity.
If there exists a residue class modulo~$3$ that contains at least $k+1$ lattice points of~$K$, then we have seen that $\LE(\inter K)\geq3k+1$, clearly a contradiction.
This means that each of the~$3^n$ residue classes contains exactly $k$ lattice points of~$K$.
Let $U$ be the set of the~$k$ lattice points of~$K$ belonging to the residue class~$3\Z^n$.

\medskip
\noindent{\it Case 1.1: The points in $U$ are non-collinear.}

From the $0$-symmetry of~$K$, we see that $U$ must be a $0$-symmetric point set.
By \cref{thm_number_vectors}, we have $|U-U|\geq3k-2$ and in fact we must have equality as $\LE(\inter K)=3k-2$.
Thus, the set~$U$ either is an incomplete arithmetic progression of type $(\lceil k/2 \rceil,2)$, or an arithmetic progression of type~$(3,3)$ (cf.~\cref{fig_equality_cases}).
The first situation cannot occur, since no incomplete arithmetic progression of type $(\lceil k/2 \rceil,2)$ can be translated as to be $0$-symmetric.
In the latter situation, we have $k=9$, and the origin is the central point in $U$.
The lattice points in the relative interior of $\conv U$ are interior lattice points of $K$, but it may happen that $\conv U$ contains exactly $3k-2=25$ relative interior lattice points.
In this case, all of the eight other residue classes with a point in the lattice plane $\lin U$ have less than~$k$ lattice points in $\conv U$.
By assumption, fixing one of these residue classes~$R$, there must be some lattice point in $K$, contained in the class $R$, which does not lie in $\conv U$.
Then, either one of the eight points $U\setminus\{0\}$ is an interior lattice point of~$K$, or the class $R$ generates an interior lattice point that is not contained in $\lin U$.
In both cases, we get more than the assumed $3k-2$ interior lattice points in~$K$ and thus a contradiction.

\medskip
\noindent{\it Case 1.2: The points in $U$ are contained in a line~$L$.}

Let $V$ be the set of the~$k$ lattice points of~$K$ contained in one of the other two residue classes that have points in~$L$.
Since $|U|=k$, the line~$L$ contains at least $k-1$ lattice points of~$V$.
Indeed, it contains exactly $k-1$ such points, since otherwise $K\cap L$ contains at least $3k-2$ interior lattice points and hence the interior lattice points of~$K$ would be collinear contradicting the assumption.
This means, that there is one point of $V$ outside ~$L$.
Since $k-1\geq2$, there are at least three pairwise linearly independent vectors in the difference set $V-V$.
Hence, there must also be three pairwise linearly independent interior lattice points of~$K$.
But this is a contradiction, because $K\cap L$ contains $3k-4$ interior lattice points of~$K$ and thus there can only be one pair of opposite interior lattice points outside ~$L$.

Concluding Case 1, we have proved that there is no convex body $K\in\K^n_o$ with non-collinear interior lattice points and that attains equality in~\eqref{eqn_main_theorem}.

\medskip
\noindent{\it Case 2: The interior lattice points of $K$ are collinear.}

In the case $\LE(\inter K)=1$ equality in~\eqref{eqn_main_theorem} has been characterized in~\cite{draismamcallisternill2012lattice}.
In fact, there is a unimodular transformation~$A$ such that $AK=C_n$.
Therefore, we assume that $\ell=\LE(\inter K)\geq 3$, and we let~$L$ be the line containing the interior lattice points of~$K$.
The lattice points of $K$ are partitioned into~$3^n$ sets each of which containing only lattice points of a fixed residue class modulo~$3$.
Let these sets be labeled~$R_{ij}$, for $i=1,\dots,3^{n-1}$ and $j=1,2,3$, such that for every lattice line~$L'$ parallel to $L$ there is some $i\in\{1,\ldots,3^{n-1}\}$ so that $L'$ contains only points of $R_{ij}$, for $j=1,2,3$.
As observed in the first paragraph of the proof, each $R_{ij}$ is contained in a line parallel to~$L$.
Moreover, if $|R_{ij}|\geq 2$, then between two consecutive points in $R_{ij}$ there are another two lattice points corresponding to the other two residue classes present on the line containing~$R_{ij}$.
Hence, $R_{i1}, R_{i2}$ and $R_{i3}$ have to be contained in the same line, and by the $0$-symmetry of~$K$ this line contains at most $\ell+2=\LE(\inter K)+2$ lattice points of~$K$ in total.
Note, that it cannot happen that there is an index $i\in\{1,\ldots,3^{n-1}\}$ such that $|R_{ij}|\leq 1$, for all $j=1,2,3$, because then
\[\LE(K)=\sum_{k=1}^{3^{n-1}}\sum_{j=1}^3 |R_{kj}|\leq\sum_{k=1}^{3^{n-1}-1}(\ell+2)+3<3^{n-1}(\ell+2)=\LE(K),\]
which is a contradiction.
Thus, for every $i\in\{1,\ldots,3^{n-1}\}$ there is a line $L_i$ parallel to~$L$ containing all the points in $R_{ij}$, $j=1,2,3$, and this line must contain exactly $\ell+2$ lattice points of~$K$ in order for $K$ to attain equality in~\eqref{eqn_main_theorem}.
Since~$\ell$ is odd, for any $i=1,\ldots,3^{n-1}$, we can write $L_i=z_i+L$, where $z_i\in K\cap\Z^n$ is the midpoint of $L_i\cap K\cap \Z^n$.
Without loss of generality, we let $z_1=0$, and hence $L_1=L$.

We now claim that $H=\lin\{z_i:\,i=1,\dots,3^{n-1}\}$ is
$(n-1)$-dimensional. As observed before, $|R_{ij}|\geq 1$, for all
$i=1,\dots,3^{n-1}$ and $j=1,2,3$. Therefore,
$K\cap\Z^n=\cup_{i,j}R_{ij}$ contains a lattice point from each of
the $3^n$ residue classes and thus is full-dimensional. Hence the
orthogonal projection of $K\cap\Z^n$ onto the lattice hyperplane
orthogonal to $L$ is an $(n-1)$-dimensional set. This implies that
$H$ has at least $n-1$ dimensions. Furthermore, the subspace $H$
cannot be full-dimensional. To see this, we let $v\in
L\cap\Z^n\setminus\{0\}$ be of minimal length and we relabel
the~$z_i$ such that $z_2,\ldots,z_n$ are linearly independent. Since
every $L_i$ contains exactly~$\ell+2$ lattice points of~$K$, the
lattice points $((\ell+1)/2)v\pm z_i$, $i=1,\ldots,n$, are contained
in~$K$. Moreover, $((\ell+1)/2)v$ is contained in the relative
interior of the $(n-1)$-dimensional crosspolytope
$\conv\{((\ell+1)/2)v\pm z_i:\,i=2,\ldots,n\}$, but it is not an
interior lattice point of~$K$.
In order to prove the claim, let $H'=\lin\{z_2,\ldots,z_n\}$ and
assume, for the sake of contradiction,
that $z_k\notin H$, for some $k\in\{n+1,\ldots,3^{n-1}\}$. Hence,
$z_k=u+\lambda v$, for some $u\in H'$. Possibly replacing $v$ by $-v$,
we assume that~$\lambda>0$. Then, $((\ell+1)/2)v+z_k$ is strictly
separated from the origin by the hyperplane $((\ell+1)/2)v+H'$. As a
consequence, we have
\[((\ell+1)/2)v\in\inter\conv\{((\ell+1)/2)v\pm z_i:i\in\{1,\ldots,n\}\cup\{k\}\}\subset\inter K,\]
which is a contradiction. So we have $H'=H$, proving the claim.

Since $K$ is convex, $Q:=\conv\bigcup_{i,j}R_{ij}\subseteq K$.
Moreover,~$Q$ is a prism with basis parallel to $H$ and length $(\ell+1)\|v\|$ in the direction $L=\lin\{v\}$.
Observe that none of the lines $z_i+L$, for $i>1$, intersects~$K$ nor~$Q$ in the interior, because otherwise we would find interior lattice points of $K$ outside~$L$.
In order to finish the proof, let us consider $Q'=\conv\{z_i\pm v:i=1,\ldots,3^{n-1}\}$.
We clearly have $\LE(\inter Q')=1$, $\LE(Q')=3^n$, and $v$ lies on the relative interior
of the facet $\conv\{v+ z_i:i\in\{1,\ldots,3^{n-1}\}\}$ of $Q'$,
and thus by the already mentioned equality characterization
in~\cite{draismamcallisternill2012lattice} there exists a unimodular transformation $A$ such that $AQ'=C_n$.
The lattice point $Av$ lies on the relative interior of a facet of $AQ'$, and
since~$A$ maps parallel lines to parallel lines, up to a suitable linear permutation
of the components, we have $AQ=C_{n-1}\times[-(\ell+1)/2,(\ell+1)/2]$.

What is left to show is that~$K=Q$.
Assume for contradiction that this is not the case, and let $p\in K\backslash Q$.
There exists a hyperplane parallel to a facet~$F$ of~$Q$ that strictly separates~$p$ from~$Q$.
Since either $e_i$ or $-e_i$ is the center of a given facet of~$C_n$, where $e_i$ denotes the $i$th coordinate unit vector, the center~$w$ of the facet~$F$ is one of the points $\pm A^{-1}e_1,\ldots,\pm A^{-1}e_{n-1}$, or $\pm A^{-1}((\ell+1)/2)e_n$.
Therefore, $\conv(\{p\}\cup Q)\subset K$ contains~$w$ in its interior, which contradicts either the fact that~$L$ contains exactly~$\ell$ interior lattice points of~$K$, or the assumption that the interior lattice points of~$K$ are collinear.
Hence, in fact $K=Q$ and we are done.
\end{proof}

\section{A connection to Helly numbers of families of \texorpdfstring{$S$}{S}-convex sets}\label{sect_helly_number_connection}

In this section, we are concerned with how much the bound in \cref{thm_main_discrete_mink} can be improved under the assumption that the involved convex body is strictly convex.
We start with the inequality~\eqref{eqn_osymm_theorem_sc} whose proof is a direct extension of Minkowski's original ideas that led to~\eqref{eqn_first_mink_disc_sc} (cf.~\cite[p.~79/80]{minkowski1896geometrie}).

\begin{proof}[Proof of \eqref{eqn_osymm_theorem_sc}]
Assume that $\LE(K)\geq2^{n+1}k$ for some $k\in\N$.
Then, $K$ contains, besides the origin, at least $2^nk$ pairs of lattice points $x,-x$.
If at least~$k$ of these pairs are congruent to $0$ modulo $2$, then the points $\pm\frac12x$ are interior lattice points of $K$ and hence $\LE(\inter K)\geq2k+1$.
In the case that there are at most $k-1$ of these pairs that are congruent to $0$ modulo $2$, the pigeon hole principle provides us with a different congruence class modulo~$2$ containing at least $k+1$ points from these pairs.
Let these points be $v_1,\ldots,v_{k+1}\in K\cap\Z^n$, and let $H$ be a hyperplane
supporting $\conv\{v_1,\ldots,v_{k+1}\}$ exactly in a vertex, say $v_1$.
Observe that the vectors $v_i-v_1$, for $i=2,\dots,k+1$, are pairwise
different and point to the same halfspace determined by $H$.
Therefore, even the vectors $\pm(v_i-v_1)$, for $i=2,\dots,k+1$,
are pairwise different.
Since $K$ is strictly convex, the points $\pm(v_i-v_1)/2$, $i=2,\ldots,k+1$, together with $0$,
provide $2k+1$ interior lattice points of $K$ and hence $\LE(\inter K)\geq 2k+1$.

In summary, assuming that $\LE(\inter K)=2k-1$ implies the desired bound $\LE(K)\leq2^{n+1}k-1=2^n\left(\LE(\inter K)+1\right)-1$.
\end{proof}

As we claim in \cref{thm_main_theorem_sc}, this result can be further improved by roughly a factor of $2/3$ on the right hand side and additionally the restriction to $0$-symmetric convex bodies can be removed.
In order to see this, we need to introduce some notions and concepts regarding Helly numbers of families of certain sets.
For background information on these topics we refer the reader to the works that are cited below in this section.

Let $S$ be an arbitrary non-empty subset of $\R^n$. A closed convex set
$L\subset\R^n$ is said to be \emph{maximal $S$-free} if the interior
of~$L$ is disjoint from~$S$, and~$L$ is inclusion-maximal with
respect to this property. Any subset of $\R^n$ is called
\emph{$S$-convex} if it can be written as the intersection of~$S$
and a convex set. The \emph{Helly number} $h(S)$ of the family of
$S$-convex sets is the smallest natural number~$h$ with the
following property:
\begin{align}
&\textrm{For every family }C_1,\ldots,C_m\textrm{ of }S\textrm{-convex
sets, such that}\label{eqn_hellynumber}\\
&\cap_{j=1}^mC_j=\emptyset,\textrm{ there exist }i_1,\ldots,i_h\textrm{ satisfying }\cap_{j=1}^hC_{i_j}=\emptyset.\nonumber
\end{align}
The classical theorem of Helly asserts that $h(\R^n)=n+1$ while
Doignon~\cite{doignon1973convexity} showed that $h(\Z^n)=2^n$.

Now, the \emph{facet complexity $f(S)$ of maximal $S$-free sets} is defined as the smallest natural number~$f$ such that every $n$-dimensional
maximal $S$-free set is a polyhedron with at most~$f$ facets.
Averkov~\cite[Thm.~2.1]{averkov2013Onmaximal} proved that
\begin{align}
f(S)\leq h(S).\label{eqn_f_h_ineq}
\end{align}
Furthermore, we need a Helly-type result for which we introduce the
constant $c(n,k)$ for every $n\in\N$ and $k\in\N\cup\{0\}$. It is
defined as the smallest natural number $c$ such that for any
polyhedron $P=\{x\in\R^n:Ax\leq b\}$, with $A\in\R^{m\times n}$,
$b\in\R^m$, and $k=\LE(P)$, there exists a subset $I\subseteq[m]$,
with $|I|=c$, of the rows of~$A$ fulfilling that the polyhedron
$P_I=\{x\in\R^n:A_Ix\leq b_I\}$ contains the same set of lattice
points as~$P$. Here $A_I\in\R^{|I|\times n}$ and $b_I\in\R^{|I|}$
consist of only those rows of~$A$ and~$b$, respectively, that are
indexed by~$I$.

Aliev, Bassett, De Loera, and
Loveaux~\cite[Thm.~1]{alievbassetdeloeralouveaux2015aquantitative}
proved that $c(n,k)$ is finite, and in particular,
\begin{align}
c(n,k)\leq 2(2^{n-1}-1)\lceil 2(k+1)/3\rceil+2,\quad\textrm{for
all }n\in\N,k\in\N\cup\{0\}.\label{eqn_bound_facets}
\end{align}

Now, we are well-prepared to prove the upper bound on~$\LE(K)$ for strictly convex bodies~$K$.

\begin{proof}[Proof of \cref{thm_main_theorem_sc}]
Consider $S=\Z^n\setminus\inter K$. It is known that there exists a
maximal $S$-free polyhedron $L$ with $K\subseteq L$
(see~\cite[p.~1614]{averkov2013Onmaximal}). Let $m$ be the number of
facets, that is, $(n\!-\!1)$-dimensional faces, of $L$. By
definition of~$f(S)$, we have $m\leq f(S)$, and due to
\eqref{eqn_f_h_ineq}, we get that $m\leq h(S)$. In view of
\cite[Prop.~1.2]{averkovweismantel2012transversal}, the property
\eqref{eqn_hellynumber} is equivalent to:
\begin{align*}
&\textrm{For every }A\in\R^{m\times n}\textrm{ and }b\in\R^m\textit{
either }P=\{x\in\R^n:Ax\leq b\}\textrm{ contains}\\
&\textrm{a point of }S\textit{ or }\textrm{there is a subset
}I\subseteq[m],\,|I|=h,\textrm{ such that}\\
&P_I=\{x\in\R^n:A_Ix\leq b_I\}\textrm{ does not contain a point of
}S.
\end{align*}
This means that for every such polyhedron $P$ with $\Z^n\cap
P=\Z^n\cap\inter K$ there is a subset $I\subseteq[m]$, $|I|=h(S)$,
such that $\Z^n\cap P_I=\Z^n\cap\inter K$. Therefore $h(S)\leq
c(n,\LE(\inter K))$ which together with~\eqref{eqn_bound_facets}
proves that
\[m\leq 2(2^{n-1}-1)\lceil 2(\LE(\inter K)+1)/3\rceil+2.\]
We observe that $S\cap K$ consists of $\LE(K)-\LE(\inter K)$ elements, and that every element of $S\cap K$ lies in some facet of $L$.
Since~$K$ is strictly convex and $S$-free, no two distinct elements of $S\cap K$ lie in the same facet of~$L$.
Therefore,~$L$ has at least $\LE(K)-\LE(\inter K)$ facets, from which we conclude that
\[\LE(K)-\LE(\inter K)\leq m\leq 2(2^{n-1}-1)\lceil 2(\LE(\inter K)+1)/3\rceil+2.\qedhere\]
\end{proof}

\begin{remark}\
\begin{enumerate}[i)]
 \item It is clear that the question whether the bound~\eqref{eqn_main_theorem_sc} is best possible is closely connected to the determination of the minimal value for $c(n,k)$.
 An extended discussion of equality cases for the inequality in~\eqref{eqn_bound_facets} can be found in~\cite{alievbassetdeloeralouveaux2015aquantitative}.
 Therein, the authors describe a polyhedron showing that $c(n,1)=2(2^{n-1}-1)$, which corresponds to the fact that Minkowski's bound~\eqref{eqn_first_mink_disc_sc} is optimal as described in the introduction.
 For $k=2$, we do not know whether~\eqref{eqn_main_theorem_sc} is sharp.
 The bound on $c(n,2)$ in~\eqref{eqn_bound_facets} is the same as on $c(n,1)$, and the authors of~\cite{alievbassetdeloeralouveaux2015aquantitative} indeed conjecture that this is tight.
 However, they point out that~\eqref{eqn_bound_facets} can be improved for any $k\geq3$, and hence~\eqref{eqn_main_theorem_sc} is never tight for $\LE(\inter K)\geq 3$.
 \item Combining the congruence argument in the proof of~\eqref{eqn_osymm_theorem_sc} and the difference set estimate in \cref{thm_number_vectors}, one may argue similarly as in \cref{sect_proof_disc_Mink} and prove that $\LE(K)\leq 2^{n+1}\LE(\inter K)/3+2^{n+2}$, for every strictly convex body $K\in\K^n_o$.
 In fact, this is roughly the same bound as in \cref{thm_main_theorem_sc} albeit this approach is limited to the $0$-symmetric case.
\end{enumerate}
\end{remark}

\medskip
\noindent{\bf Acknowledgment.} The authors thank the referee for many suggestions and comments that greatly improved the paper.
In particular, we are grateful for pointing us to the connection with Helly numbers of families of $S$-convex sets and to the proof of \cref{thm_main_theorem_sc}.

\bibliographystyle{amsplain}
\bibliography{mybib}

\end{document}